\documentclass[11pt]{amsart}

\usepackage[margin=1in]{geometry}  
\usepackage{graphicx}              
\usepackage{amsmath}               
\usepackage{amsfonts}    
\usepackage{amssymb}

\newtheorem{thm}{Theorem}[section]
\newtheorem{prop}[thm]{Proposition}
\newtheorem{lem}[thm]{Lemma}
\newtheorem{cor}[thm]{Corollary}

\theoremstyle{definition}
\newtheorem{defi}[thm]{Definition}

\theoremstyle{remark}

\numberwithin{equation}{section}

\newcommand{\R}{\mathbb{R}^n}

\begin{document}

\title{A Dirichlet problem for nonlocal degenerate elliptic operators with internal nonlinearity} 
\author{Hui Yu}
\address{Department of Mathematics, the University of Texas at Austin}
\email{hyu@math.utexas.edu}

\begin{abstract}
We study a Dirichlet problem in the entire space for some nonlocal degenerate elliptic operators with internal nonlinearities. With very mild assumptions on the boundary datum, we prove existence and uniqueness of the solution in the viscosity sense. If we further assume uniform ellipticity then the solution is shown to be classical, and even smooth if both the operator and the boundary datum are smooth.
\end{abstract}

 \maketitle

\tableofcontents

\section{Introduction}
Nonlocal elliptic operators model diffusion processes with long-term interactions. Since Caffarelli and Silvestre introduced the notion of viscosity solutions for these operators \cite{CS1}, one  is able to  deal with very general classes of nonlocal fully nonlinear elliptic operators, and a theory analogue to the  case of second order elliptic equations is established. See for instance the works of Caffarelli-Silvestre \cite{CS1}\cite{CS2}\cite{CS3}, Kriventsov \cite{Kri}, Jin-Xiong \cite{JX}, Serra \cite {Ser} and Yu \cite{Y1}\cite{Y2}.

In these works a nonlocal elliptic operator is of the form \begin{equation}
I[u,x]=\inf_{\alpha}\sup_{\beta} L_{\alpha\beta}u(x),
\end{equation} where each $L_{\alpha\beta}$ is a linear operator of the form \begin{equation*}
L_{\alpha\beta}u(x)=c_{n,\sigma}\int \delta u(x,y)K_{\alpha\beta}(x,y)dy.\end{equation*} Here $c_{n,\sigma}$ is a constant depending on the dimension of the space $n$, as well as the order of the operator $\sigma$, which is always assumed to be in $(1,2)$ in this work.  $\delta u(x,y)=u(x+y)+u(x-y)-2u(x)$ is the symmetric difference centered at the point $x$. $K_{\alpha\beta}(x,y)$ are some kernels comparable to $\frac{1}{|y|^{n+\sigma}}$, which is the kernel for the classical fractional Laplacian $\Delta^{\sigma/2}$. 

In a sense, these kernels assign weights to information coming from different locations and directions in the media. Taking the place of coefficient matrices in second order equations, they encode the `inhomogeneity'  and `anisotropy' of the underlying media. 

Since an integral kernel enjoys more `degrees of freedom' than a matrix, the theory of nonlocal operators allows much richer `spatial inhomogeneity and anisotropy' in the media. However, for operators as in (1.1), the dependence on $\delta u(x,y)$ is still trivial. To cover possibly different dependence on $\delta u(x,y)$, we propose to study operators of the following form \begin{equation}
I[u,x]=\int \frac{F(\delta u(x,y))}{|y|^{n+\sigma}}dy,
\end{equation}where $F$ is an increasing function with $F(0)=0$. 

Here the underlying medium is homogeneous and isotropic as the kernel is simply the kernel for fractional Laplacian. But the dependence on $\delta u$ can take various forms. For instance, $F(t)=10^5 t\chi_{|t|<0.01}+t\chi_{0.01<|t|<100}+10^{-5}t\chi_{|t|>100}$ models a diffusion process where one sees strong diffusive effect at `near equilibrium' points but very weak diffusive effect at `far from equilibrium' points. For another example, $F(t)=10^5 t\chi_{t>0}+10^{-5}t\chi_{t<0}$ models a process where the diffusion is strong at `convex' points but is weak at `concave' points.

It is interesting to note that in the limit as $\sigma\to 2$, an operator of this form converges, at least formally, to a constant multiple of the Laplacian, the constant being $F'(0)$. This may explain why operators of this form have not received much attention. However, in the case when $\sigma<2$, they do exhibit nontrivial behaviour.

In this paper we study the following Dirichlet problem for this type of operators with `internal nonlinearity'\footnote{This term was suggested by Dennis Kriventsov.} $F$:

$$\begin{cases}
\int \frac{F(\delta u(x,y))}{|y|^{n+\sigma}}dy=g(x,u-\phi)  &\text{     in $\R$}\\ u-\phi\to 0 &\text{   at $\infty$}.
\end{cases}$$ We impose the following conditions throughout the paper on the nonlinearity $F$, the forcing term $g$ and the boundary datum $\phi$:

\begin{itemize}\item{$F: \mathbb{R}\to\mathbb{R}$ is a $C^1$ function which satisfies
\begin{equation}
F(0)=0,  
\end{equation} 
\begin{equation}
Lip(F)<L_1,
\end{equation} 
and \begin{equation}
F'>0.
\end{equation}} 
\item{$g:\R\times \mathbb{R}\to \mathbb{R}$ is Lipschitz in the first variable
\begin{equation}
|g(x,t)-g(x',t)|<L_2|x-x'|,
\end{equation}  and uniformly increasing in the second variable \begin{equation}
g(x,t)-g(x.s)\ge\mu(t-s) \text{ if $t>s$ for some $\mu>0$,}
\end{equation} and \begin{equation}
g(x,0)=0 \text{ for all $x\in\R$}.
\end{equation} }
\item{$\phi:\R\to\mathbb{R}$ is a convex Lipchitz $C^{1,1}$-function with finite $\mathcal{L}^1(\frac{1}{(1+|y|)^{n+\sigma}}dy)$ norm\begin{equation}
0\le D^2(\phi)\le L_3, 
\end{equation} 
\begin{equation}
Lip(\phi)\le L_4,
\end{equation} and \begin{equation}
\int |\phi(y)|\frac{1}{(1+|y|)^{n+\sigma}}dy\le L_5.
\end{equation}}\end{itemize} 

Let's make a few remarks. 

We are solving the Dirichlet problem in the entire space $\R$, which seems a natural first step before moving to domains with more interesting geometry. Condition (1.5) makes this operator elliptic. However no lower bound on $F'$ is assumed, and thus allows the possibility of degeneracy. To account for this lack of uniform ellipticity, we impose (1.7) on the forcing term, which will be crucial for the comparison principle and the well-posedness of the problem. Conditions (1.8) and (1.10) make sure the operator can be computed classically at $\phi$. The convexity of $\phi$ says that it is a natural subsolution.

Under these assumptions we first prove a comparison principle between viscosity subsolutions and supersolutions: \begin{thm}
Let $u,v$ be two continuos functions satisfying the following in the viscosity sense
$$\begin{cases}
\int \frac{F(\delta u(x,y))}{|y|^{n+\sigma}}dy\ge g(x,u-\phi)  &\text{     in $\R$}\\ u-\phi\to 0 &\text{   at $\infty$},
\end{cases}$$
and $$\begin{cases}
\int \frac{F(\delta v(x,y))}{|y|^{n+\sigma}}dy\le g(x,v-\phi)  &\text{     in $\R$}\\ v-\phi\to 0 &\text{   at $\infty$}. 
\end{cases}$$ Then \begin{equation*}
u\le v \text{ in $\R$}.
\end{equation*} 
\end{thm} 

This comparison principle together with the translation-invariant character of the operator also yield the following regularity estimate, which do not depend on uniform ellipticity.

\begin{thm}
Let $u$ solve the Dirichlet problem in the viscosity sense. Then $u$ is $C^{\frac{2-\sigma}{2}}$ with estimate \begin{equation}
[u]_{C^{\frac{2-\sigma}{2}}}\le C(\|u-\phi\|_{\mathcal{\infty}}+1),
\end{equation} where the constant $C$ depends on the constants in (1.4)-(1.11).
\end{thm}

The comparison principle implies the uniqueness of the solution. And the above regularity provides enough compactness for an approximation argument that gives the existence result.

\begin{thm}Suppose $\phi$ is `close to a cone' at $\infty$. Then under either of these assumptions \begin{itemize}
\item{$F$ is concave}
\end{itemize} or \begin{itemize}\item{$g$ grows superlinearly in the second variable,}\end{itemize}
there exists a unique viscosity solution to the Dirichlet problem.\end{thm} 

The extra assumptions are mainly used to construct an appropriate supersolutions. The `close to a cone' property was used already in Caffarelli-Charro \cite{CCh}, its precise definition given in section 4.  

After showing the existence and uniqueness of solution to the possibly degenerate operator, we study the regularity of the solution. For this we impose the uniform ellipticity condition on $F$. This and some conditions on $g$ and $\phi$ guarantees the solution is classical \begin{thm}
If \begin{itemize}\item{$F$ is $C^2$ with $0<\lambda <F'<\Lambda$ and $|F''|\le L_6$,}\item{$g$ is locally Lipschitz in the second variable,} and \item{$Lip(D^2(\phi))<L_7$,}
\end{itemize} then the viscosity solution $u$ is in the class $C^{1+\sigma+\alpha}$ for some universal $0<\alpha<1$, and in particular is classical.
\end{thm} 

Once the solution is shown to be classical, a bootstrap argument gives the smoothness of the solution if we assume that $F$, $g$, and $\phi$ are smooth: \begin{thm}
Let $F$, $g$ and $\phi$ be as in the previous theorem. If one further assumes that they are smooth functions, then the solution $u$ is also smooth.
\end{thm} 

We'd like to point out so far very few nonlocal fully nonlinear operators have been shown to admit smooth solutions, one of the major obstructions being the very rough behaviour of solutions near the boundary \cite{RosSer}. On the other hand, the previous theorem seems to suggest that smoothness should be expected in the absence of boundaries. A certain class of nonlocal fully nonlinear elliptic operators with smooth solutions in bounded domains was identified in \cite{Y1}.

This paper is organized in the following way: In the second section we recall several useful definitions and propositions about viscosity solutions to nonlocal operators. In the third section we prove the comparison principle using a regularization technique by Jensen \cite{Jen}. We also prove a H\"older estimate that does not depend on uniform ellipticity but follow rather easily from the translation-invariance. In the fourth section the existence result is established via Perron's method and a compactness argument. Then in the last section the regularity of the solution is studied.

Let us remark that this work leaves some interesting questions open. Firstly one might want to study the same problem in a general domain $\Omega$ where interesting geometry can happen. A good start might be to construct super- and sub-solutions. Secondly the superlinearity condition on $g$ can possibly be removed by some more careful estimates. Lastly, the conditions on $F$, $g$ and $\phi$ for regularity, especially the Lipschitz constraint of $D^2\phi$, could be excessive. They might be weaken substantially using some scaling argument similar to the one in \cite{Kri}.

\section{Preliminaries}
The following notion of viscosity solutions to nonlocal operators, first used in \cite{CS1}, is by now standard:
\begin{defi}
Let $I$ be a $\sigma$-order nonlocal operator and $f$ be continuous. 

An upper semicontinuous function $v\in\mathcal{L}^1(\frac{1}{(1+|y|)^{n+\sigma}}dy)$ is a subsolution to \begin{equation*}
I[u,x]=f(x) \text{ in $\Omega$}
\end{equation*} if for any $x\in\Omega$, and any $\eta$ that is $C^{1,1}$ in some open set $U\ni x$ with $\eta(x)=u(x)$ and $\eta\ge u$ in $U$, the following function 
$$\tilde{\eta}=\begin{cases}\eta &\text{in $U$}\\v &\text{outside $U$}\end{cases} $$ satisfies in the classical pointwise sense \begin{equation*}I[\tilde{\eta},x]\ge f(x).\end{equation*}
 
 If $v$ is a viscosity subsolution, we write 
\begin{equation*}
I[v,x]\ge f(x) \text{ in $\Omega$}.
\end{equation*}

The notion of supersolutions is defined in the obvious manner. And a solution is at the same time a subsolution and a supersolution.
\end{defi} 

As an important technical tool, the following notion of sup- and inf-envelopes was first introduced by Jensen \cite{Jen}. 
\begin{defi}
Let $u$ be a bounded continuous function in $\Omega$. For each $\epsilon>0$, the $\epsilon$-sup-envelope of $u$ is \begin{equation*}
u^{\epsilon}(x)=\sup_{y\in\Omega}(u(y)-\frac{|y-x|^2}{\epsilon}).
\end{equation*} The $\epsilon$-inf-envelope of $u$ is 
\begin{equation*}
u_{\epsilon}(x)=\inf_{y\in\Omega}(u(y)+\frac{|y-x|^2}{\epsilon}).
\end{equation*}
\end{defi} 

The following properties of sup-envelopes are elementary:
\begin{prop}
For $\epsilon>0$, let $u^\epsilon$ be the $\epsilon$-sup-envelope of a continuous bounded function $u$. Then for each $x$ there is $x^{\epsilon,*}$ such that \begin{equation*}
u^{\epsilon}(x)=u(x^{\epsilon,*})-\frac{|y-x^{\epsilon,*}|^2}{\epsilon}.
\end{equation*} Moreover $|x^{\epsilon,*}-x|\to 0$ as $\epsilon\to 0$.

Also, at each point $x$ there is a parabola with opening $\frac{1}{2\epsilon}$ touching $u^{\epsilon}$ from below at $x$.
\end{prop} 

\begin{proof}Let $x$ be a point in space.

If $|y-x|^2>2\epsilon\|u\|_{\mathcal{L}^{\infty}},$ then  $$u(y)-\frac{|y-x|^2}{\epsilon}\le u(y)-2\|u\|_{\mathcal{L}^{\infty}}\le -\|u\|_{\mathcal{L}^{\infty}}.$$ Consequently 
$$\sup_{y\in\Omega}(u(y)-\frac{|y-x|^2}{\epsilon})=\sup_{|y-x|^2\le 2\epsilon\|u\|_{\mathcal{L}^{\infty}}}(u(y)-\frac{|y-x|^2}{\epsilon}).$$ 

The continuity of $u$ and the compactness of $\{|y-x|^2\le 2\epsilon\|u\|_{\mathcal{L}^{\infty}}\}$ guarantee the existence of a maximizer $x^{\epsilon,*}$. It also gives the convergence of $x^{\epsilon, *}\to x$ since $$|x^{\epsilon,*}-x|\le 2\|u\|_{\mathcal{L}^{\infty}}\epsilon.$$

For each $x$, we define a parabola of opening $\frac{1}{2\epsilon}$ $$P(z)=u(x^{\epsilon,*})-\frac{|z-x^{\epsilon, *}|^2}{\epsilon}.$$ 

Then $P(x)=u^{\epsilon}(x)$ by definition of $x^{\epsilon, *}$. Also \begin{align*}u^{\epsilon}(z)&=\sup_{y}(u(y)-\frac{|y-z|^2}{\epsilon})\\&\ge u(x^{\epsilon,*})-\frac{|x^{\epsilon,*}-z|^2}{\epsilon}\\&=P(z).\end{align*}

Consequently $P$ touches $u^\epsilon$ from below at $x$.
\end{proof} 
Similar properties hold for the inf-envelopes. 

Throughout this paper, $x^{\epsilon, *}$ is a point where the value  $u^{\epsilon}(x)$ is realized in the sup. $x_{\epsilon, *}$ is a point where the value $u_{\epsilon}(x)$ is realized in the inf. 

The following proposition states a general phenomenon that at points of contact the equation is satisfied classically. We prove a version for our operator.

\begin{prop}
Suppose $u$ satisfies in the viscosity sense $$\begin{cases}\int \frac{F(\delta u(x,y))}{|y|^{n+\sigma}}dy\ge f(x) &\text{in $\R$}\\ u-\phi\to 0 &\text{at $\infty$},\end{cases}$$ and $\eta$ is a $C^{1,1}$ function that touches $u$ from above at $x_0$ in $U$. Then the integral is well-defined at $x_0$ with $$\int \frac{F(\delta u(x_0,y))}{|y|^{n+\sigma}}dy\ge f(x_0)$$ classically.
\end{prop} 

\begin{proof}
Take $r_0>0$ small so that $B_{r_0}(x_0)\subset U$. Define for each $0<r\le r_0$ 
$$\eta_{r}=\begin{cases}\eta &\text{in $B_r(x_0)$}\\ u &\text{outside $B_r(x_0)$}.\end{cases}$$

Then for $y\in B_{r_0/2}$ the $C^{1,1}$ regularity of $\eta$ gives $$|F(\delta \eta_{r_0}(x,y))|\le C|y|^2$$ and hence $$\int_{B_{r_0/2}} \frac{|F(\delta \eta_{r_0}(x,y))|}{|y|^{n+\sigma}}dy<\infty.$$ Meanwhile  for $R$ sufficiently large (depending on $x_0$), $|u(x_0+y)-\phi(x_0+y)|<1$ and $|u(x_0-y)-\phi(x_0-y)|<1$. 

Hence
 $$\int_{B_{R}^c} \frac{|F(\delta \eta_{r_0}(x,y))|}{|y|^{n+\sigma}}dy\le C(\|\phi\|_{\mathcal{L}^1(\frac{1}{(1+|y|)^{n+\sigma}})dy}+1).$$ Consequently $$\int \frac{|F(\delta \eta_{r_0}(x,y))|}{|y|^{n+\sigma}}dy<\infty.$$

Also by the subsolution property $$ \int \frac{F(\delta \eta_r(x_0,y))}{|y|^{n+\sigma}}dy\ge f(x_0) \text{ for all $r$}.$$

Now note that for $0<r_2<r_1\le r_0$, one has $\eta_{r_2}\le \eta_{r_1}$ and $\eta_{r_2}(x_0)= \eta_{r_1}(x_0)$. As a result $\delta \eta_{r_2}(x_0,\cdot)\le \delta \eta_{r_1}(x_0,\cdot), $ and $F(\delta \eta_{r_2}(x_0,\cdot))\le F(\delta \eta_{r_1}(x_0,\cdot))$. With $\tilde{\eta}_r\to u$ pointwisely, monotone convergence theorem gives the finiteness of the integral for $u$ and the inequality in the classical sense.
\end{proof} 
Again similar property holds for supersolutions.

\section{Comparison principle and H\"older regularity}

We prove the following comparison principle for viscosity sub- and supersolutions to our operator. Although we are dealing with a nonlocal operator with degeneracy, the idea is essentially the same as in \cite{CC}.  For a very similar argument in the nonlocal setting, see Caffarelli-Charro \cite{CCh}.

\begin{thm}
Let $u,v$ be two continuos functions satisfying the following in the viscosity sense
$$\begin{cases}
\int \frac{F(\delta u(x,y))}{|y|^{n+\sigma}}dy\ge g(x,u-\phi)  &\text{     in $\R$}\\ u-\phi\to 0 &\text{   at $\infty$},
\end{cases}$$
and $$\begin{cases}
\int \frac{F(\delta v(x,y))}{|y|^{n+\sigma}}dy\le g(x,v-\phi)  &\text{     in $\R$}\\ v-\phi\to 0 &\text{   at $\infty$}. 
\end{cases}$$ Then \begin{equation*}
u\le v \text{ in $\R$}.
\end{equation*} 
\end{thm} 

As a simple corollary, we have the following: 
\begin{cor}The viscosity solution to the Dirichlet problem, if exists, is unique.
\end{cor}

Now we prove the theorem.
\begin{proof} Suppose, on the contrary, that $u>v$ at some point in $\R$.

Define $U=u-\phi$, $V=v-\phi$, and $U^\epsilon$ the $\epsilon$-sup-envelope of $U$, $V_{\epsilon}$ the $\epsilon$-inf-envelope of $V$. 

The trivial inequalities $U^{\epsilon}\ge u-\phi$ and $V_{\epsilon}\le v-\phi$ give some point $x$ where $U^{\epsilon}(x)-V_{\epsilon}(x)>\delta$ for some small $\delta>0$. 

Note this $\delta$ is independent of $\epsilon$.

Now for $y\in\R$ 
\begin{align*}
U^{\epsilon}(y)-V_{\epsilon}(y)&=(U(y^{\epsilon,*})-\frac{|y^{\epsilon,*}-y|^2}{\epsilon})-(V(y_{\epsilon,*})+\frac{|y_{\epsilon,*}-y|^2}{\epsilon})\\ &\le U(y^{\epsilon,*})-V(y_{\epsilon,*}). 
\end{align*} When $y\to\infty$, both $y^{\epsilon,*}$ and $y_{\epsilon,*}$ also go to infinity since they stay close to $y$. Hence the last term converges to $0$. As a result the maximum of $U^{\epsilon}-V_{\epsilon}$ is realized in some bounded region. By continuity there exists some $\bar{x}$ such that $$U^{\epsilon}(\bar{x})-V_{\epsilon}(\bar{x})=\max (U^{\epsilon}-V_{\epsilon})>\delta.$$

Now we show that at the point $\bar{x}$, $U^\epsilon$ and $V_\epsilon$ are $C^{1,1}$. To see this , note that for any $x$
\begin{equation*}
V_{\epsilon}(x)+(U^{\epsilon}(\bar{x})-V_{\epsilon}(\bar{x}))
\ge U^{\epsilon}(x)\end{equation*} with equality at $\bar{x}$. Meanwhile, Proposition 2.3 gives a parabola with opening $\frac{1}{2\epsilon}$ touching $V_{\epsilon}$ from above at $\bar{x}$. The same parabola, shifted by $(U^{\epsilon}(\bar{x})-V_{\epsilon}(\bar{x}))$, touches $U^{\epsilon}$ from above at $\bar{x}$. Proposition 2.3 also gives a parabola touching $U^{\epsilon}$ from above at $\bar{x}$. Therefore $U^{\epsilon}$ is $C^{1,1}$ at $\bar{x}$ with constant $\frac{1}{2\epsilon}$. 

Similar argument applies to $V_\epsilon$.

This regularity is inherited, from one side, by $u$ and $v$ at $\bar{x}^{\epsilon, *}$ and $\bar{x}_{\epsilon, *}$ respectively. For any $P$ touching $U^{\epsilon}$ from above at $\bar{x}$, the shifted $$\tilde{P}(x)=P(x-(\bar{x}^{\epsilon,*}-\bar{x}))+\frac{|\bar{x}^{\epsilon, *}-\bar{x}|^2}{\epsilon}$$ satisfies
\begin{align*}
\tilde{P}(x)&\ge U^{\epsilon}(x-(\bar{x}^{\epsilon,*}-\bar{x}))+\frac{|\bar{x}^{\epsilon, *}-\bar{x}|^2}{\epsilon}\\&\ge U(x)-\frac{|x-x+(\bar{x}^{\epsilon,*}-\bar{x})|^2}{\epsilon}+\frac{|\bar{x}^{\epsilon, *}-\bar{x}|^2}{\epsilon}\\&=U(x).
\end{align*} Also \begin{align*}\tilde{P}(\bar{x}^{\epsilon,*})&=P(\bar{x})+\frac{|\bar{x}^{\epsilon, *}-\bar{x}|^2}{\epsilon}\\&=U^{\epsilon}(\bar{x})+\frac{|\bar{x}^{\epsilon, *}-\bar{x}|^2}{\epsilon}\\&=U(\bar{x}^{\epsilon,*})-\frac{|\bar{x}^{\epsilon, *}-\bar{x}|^2}{\epsilon}+\frac{|\bar{x}^{\epsilon, *}-\bar{x}|^2}{\epsilon}\\&=U(\bar{x}^{\epsilon,*}).\end{align*} Thus $\tilde{P}$ touches $U$ from above at $\bar{x}^{\epsilon,*}$. Therefore $\tilde{P}+\phi$ touches $u$ from above at the same point. 

Similarly a $C^{1,1}$ function with the same structure touches $v$ from below at $\bar{x}_{\epsilon,*}.$
This is enough, by Proposition 2.4, for the subsolution and supersolution properties to be satisfied in the classical pointwise sense.

Now note that \begin{align*}\delta U^{\epsilon}(\bar{x},y)&=U^{\epsilon}(\bar{x}+y)+U^{\epsilon}(\bar{x}-y)-2U^{\epsilon}(\bar{x})\\&\ge (U(\bar{x}^{\epsilon,*}+y)-\frac{|\bar{x}^{\epsilon,*}+y-(\bar{x}+y)|^2}{\epsilon})+(U(\bar{x}^{\epsilon,*}-y)-\frac{|\bar{x}^{\epsilon,*}-y-(\bar{x}-y)|^2}{\epsilon})\\&-2(U(\bar{x}^{\epsilon,*})-\frac{|\bar{x}^{\epsilon,*}-\bar{x}|^2}{\epsilon})\\&\ge (u(\bar{x}^{\epsilon,*}+y)-\phi(\bar{x}^{\epsilon,*}+y)-\frac{|\bar{x}^{\epsilon,*}-\bar{x}|^2}{\epsilon})+(u(\bar{x}^{\epsilon,*}-y)-\phi(\bar{x}^{\epsilon,*}-y)-\frac{|\bar{x}^{\epsilon,*}-\bar{x}|^2}{\epsilon})\\&-2(u(\bar{x}^{\epsilon,*})-\phi(\bar{x}^{\epsilon,*})-\frac{|\bar{x}^{\epsilon,*}-\bar{x}|^2}{\epsilon})\\&=\delta u(\bar{x}^{\epsilon, *},y)-\delta\phi(\bar{x}^{\epsilon, *},y).
\end{align*} Similarly $$\delta V_{\epsilon}(\bar{x},y)\le \delta v(\bar{x}_{\epsilon, *},y)-\delta\phi(\bar{x}_{\epsilon, *},y).$$
Now $\bar{x}$ being a maximum point for $U^{\epsilon}-V_{\epsilon}$, one has $$\delta(U^\epsilon-V_{\epsilon})(\bar{x},y)\le 0,$$ which yields  $$\delta U^\epsilon(\bar{x},y)\le\delta V_{\epsilon}(\bar{x},y).$$
Combining these three inequalities one obtains
\begin{equation}\delta u(\bar{x}^{\epsilon,*},y)\le \delta v(\bar{x}_{\epsilon, *},y)-\delta\phi(\bar{x}_{\epsilon, *},y)+\delta\phi(\bar{x}^{\epsilon, *},y).\end{equation} 

By monotonicity of $F$ and the fact that we have inequalities in the classical sense, one has \begin{align*}
\int\frac{F(\delta u(\bar{x}^{\epsilon,*},y))}{|y|^{n+\sigma}}dy&\le \int\frac{F(\delta v(\bar{x}_{\epsilon, *},y)-\delta\phi(\bar{x}_{\epsilon, *},y)+\delta\phi(\bar{x}^{\epsilon, *},y))}{|y|^{n+\sigma}}dy\\&\le  \int\frac{F(\delta v(\bar{x}_{\epsilon, *},y))}{|y|^{n+\sigma}}dy+L_1 \int\frac{|\delta\phi(\bar{x}_{\epsilon, *},y)-\delta\phi(\bar{x}^{\epsilon, *},y)|}{|y|^{n+\sigma}}dy.
\end{align*}As a result,
\begin{align*}
L_1\int\frac{|\delta\phi(\bar{x}_{\epsilon, *},y)-\delta\phi(\bar{x}^{\epsilon, *},y)|}{|y|^{n+\sigma}}dy&\ge\int\frac{F(\delta u(\bar{x}^{\epsilon,*},y))}{|y|^{n+\sigma}}dy-\int\frac{F(\delta v(\bar{x}_{\epsilon, *},y))}{|y|^{n+\sigma}}dy\\&\ge g(\bar{x}^{\epsilon,*},u(\bar{x}^{\epsilon,*})-\phi(\bar{x}^{\epsilon,*}))- g(\bar{x}_{\epsilon,*},v(\bar{x}_{\epsilon,*})-\phi(\bar{x}_{\epsilon,*}))\\&= g(\bar{x}^{\epsilon,*},U(\bar{x}^{\epsilon,*}))- g(\bar{x}_{\epsilon,*},V(\bar{x}_{\epsilon,*}))\\&=  g(\bar{x}^{\epsilon,*},U^{\epsilon}(\bar{x})+\frac{|\bar{x}-\bar{x}^{\epsilon,*}|^2}{\epsilon})- g(\bar{x}_{\epsilon,*},V_{\epsilon}(\bar{x})-\frac{|\bar{x}-\bar{x}_{\epsilon,*}|^2}{\epsilon})\\&\ge g(\bar{x}^{\epsilon,*},U^{\epsilon}(\bar{x}))- g(\bar{x}_{\epsilon,*},V_{\epsilon}(\bar{x}))\\&\ge g(\bar{x}_{\epsilon,*},U^{\epsilon}(\bar{x}))- g(\bar{x}_{\epsilon,*},V_{\epsilon}(\bar{x}))-L_2|\bar{x}_{\epsilon,*}-\bar{x}^{\epsilon,*}|\\&\ge \mu\delta -L_2|\bar{x}_{\epsilon,*}-\bar{x}^{\epsilon,*}|.
\end{align*}
Now note that when $\epsilon\to 0$, the left hand side is of order $o(1)$. To see this, note that

\begin{align*}\frac{|\delta\phi(\bar{x}_{\epsilon, *},y)-\delta\phi(\bar{x}^{\epsilon, *},y)|}{|y|^{n+\sigma}}&\le\frac{|\delta\phi(\bar{x}_{\epsilon, *},y)|+|\delta\phi(\bar{x}^{\epsilon, *},y)|}{|y|^{n+\sigma}}\\&\le 2L_3\frac{|y|^2}{|y|^{n+\sigma}}\chi_{B_{10}}(y)+8\frac{|\phi(y)|}{|y|^{n+\sigma}}\chi_{B_{10}^c}(y)
\end{align*}due to the $C^{1,1}$ regularity of $\phi$ and its $\mathcal{L}^1(\frac{1}{(1+|y|)^{n+\sigma}}dy)$ integrability, the last term in an integrable function.. Thus $\bar{x}^{\epsilon,*}\to\bar{x}$ and $\bar{x}_{\epsilon,*}\to\bar{x}$ together with the dominated convergence theorem gives $$\int\frac{|\delta\phi(\bar{x}_{\epsilon, *},y)-\delta\phi(\bar{x}^{\epsilon, *},y)|}{|y|^{n+\sigma}}dy=o(1).$$


Also,
$L_2|\bar{x}_{\epsilon,*}-\bar{x}^{\epsilon,*}|=o(1)$. Therefore the previous inequality leads to $$\mu\delta\le o(1),$$ a contradiction.
\end{proof} 

Once a comparison principle is established, the solution naturally inherits some first regularity from the boundary datum. We show here that the H\"older seminorm of a solution is controlled. \begin{thm}
Let $u$ solve the Dirichlet problem in the viscosity sense. Then $u$ is $C^{\frac{2-\sigma}{2}}$ with estimate \begin{equation}
[u]_{C^{\frac{2-\sigma}{2}}}\le C(\|u-\phi\|_{\mathcal{\infty}}+1),
\end{equation} where the constant $C$ depends on the constants in (1.4)-(1.11).
\end{thm} 

This theorem clearly follows from the following estimate at small-scales.
\begin{lem}
Let $u$ be as in Theorem 3.3, $e\in\mathbb{S}^{n-1}$ and $|h|\le 1$, then \begin{equation} 
|u(x+he)-u(x)|\le C|h|^{\frac{2-\sigma}{2}},
\end{equation} where the constant $C$ depends on $L_2, L_3, L_4$ and $\mu$.
\end{lem} 

\begin{proof}
Suppose, on the contrary, that for some $x_0\in\R, e\in\mathbb{S}^{n-1}$ and $0<h<1$ one has $$u(x_0+he)-u(x_0)>(C+1)L_4h^{\frac{2-\sigma}{2}},$$where $C$ is a large constant to be chosen. Then $$(u(x_0+he)-\phi(x_0+he))-(u(x_0)-\phi(x_0))>(C+1)L_4h^{\frac{2-\sigma}{2}}-L_4h>CL_4h^{\frac{2-\sigma}{2}}.$$

Define $U(\cdot)=(u-\phi)(\cdot+he)$ and $V(\cdot)=(u-\phi)(\cdot)$, and $U^\epsilon$, $V_{\epsilon}$ the $\epsilon$-sup- and $\epsilon$-inf-envelope of $U$ and $V$ respectively. Then $U^{\epsilon}\ge U$ and $V_{\epsilon}\le V$ gives $$\sup(U^\epsilon-V_\epsilon)>CL_4h^{\frac{2-\sigma}{2}}.$$

Note that both $U^{\epsilon}$ and $V_\epsilon$ vanish at infinity, and hence we can find $\bar{x}\in\R$ such that $$(U^\epsilon-V_\epsilon)(\bar{x})= \sup(U^\epsilon-V_\epsilon).$$

Due to maximality, at this point $$\delta (U^\epsilon-V_\epsilon)(\bar{x},\cdot)\le 0.$$ As a result $\delta U^\epsilon(\bar{x},\cdot)\le \delta V_\epsilon(\bar{x},\cdot)$, and with similar arguments as for the proof of (3.1) one obtains
\begin{align*}\delta u(\bar{x}_{\epsilon,*},y)-\delta\phi(\bar{x}_{\epsilon,*},y)&=\delta V(\bar{x}_{\epsilon,*},y)\\&\ge \delta V_\epsilon(\bar{x},y)\\&\ge \delta U^\epsilon(\bar{x},y)\\&\ge \delta U(\bar{x}^{\epsilon,*},y)\\&=\delta u(\bar{x}^{\epsilon,*}+he,y)-\delta\phi(\bar{x}^{\epsilon,*}+he,y).
\end{align*} That is,$$\delta u(\bar{x}_{\epsilon,*},y)-\delta\phi(\bar{x}_{\epsilon,*},y)+\delta\phi(\bar{x}^{\epsilon,*}+he,y)\ge \delta u(\bar{x}^{\epsilon,*}+he,y).$$

Again using similar arguments as in the previous theorem, we can show the sub- and super-solution properties are satisfied in the classical sense at points $\bar{x}^{\epsilon,*}+he$ and $\bar{x}_{\epsilon,*}$ respectively. As a result,
\begin{align*}
g(\bar{x}^{\epsilon,*}+he, (u-\phi)(\bar{x}^{\epsilon,*}+he))&\le \int\frac{F(\delta u(\bar{x}^{\epsilon,*}+he,y))}{|y|^{n+\sigma}}dy\\&\le \int\frac{F(\delta u(\bar{x}_{\epsilon,*},y)-\delta\phi(\bar{x}_{\epsilon,*},y)+\delta\phi(\bar{x}^{\epsilon,*}+he,y))}{|y|^{n+\sigma}}dy\\&\le \int\frac{F(\delta u(\bar{x}_{\epsilon,*},y))}{|y|^{n+\sigma}}dy+L_1\int\frac{|\delta\phi(\bar{x}_{\epsilon,*},y)-\delta\phi(\bar{x}^{\epsilon,*}+he,y)|}{|y|^{n+\sigma}}dy\\&\le \int\frac{F(\delta u(\bar{x}_{\epsilon,*},y))}{|y|^{n+\sigma}}dy+L_1\int\frac{|\delta\phi(\bar{x}_{\epsilon,*},y)-\delta\phi(\bar{x}^{\epsilon,*},y)|}{|y|^{n+\sigma}}dy\\&+L_1\int\frac{|\delta\phi(\bar{x}^{\epsilon,*},y)-\delta\phi(\bar{x}^{\epsilon,*}+he,y)|}{|y|^{n+\sigma}}dy\\&\le g(\bar{x}_{\epsilon,*},(u-\phi)(\bar{x}_{\epsilon,*}))+L_1\int\frac{|\delta\phi(\bar{x}_{\epsilon,*},y)-\delta\phi(\bar{x}^{\epsilon,*},y)|}{|y|^{n+\sigma}}dy\\&+L_1\int\frac{|\delta\phi(\bar{x}^{\epsilon,*},y)-\delta\phi(\bar{x}^{\epsilon,*}+he,y)|}{|y|^{n+\sigma}}dy.
\end{align*}
Using the definition for the points $\bar{x}^{\epsilon,*}$ and $\bar{x}_{\epsilon,*}$, one further obtains
\begin{align*}g(\bar{x}^{\epsilon,*}+he, U^{\epsilon}(\bar{x}))&\le g(\bar{x}^{\epsilon,*}+he, U^{\epsilon}(\bar{x})+\frac{|\bar{x}-\bar{x}^{\epsilon,*}|^2}{\epsilon})\\&= g(\bar{x}^{\epsilon,*}+he, U(\bar{x}^{\epsilon,*}))\\&=g(\bar{x}^{\epsilon,*}+he, (u-\phi)(\bar{x}^{\epsilon,*}+he)).
\end{align*}And $$g(\bar{x}_{\epsilon,*}, (u-\phi)(\bar{x}_{\epsilon,*}))\ge g(\bar{x}_{\epsilon,*}, V_\epsilon(\bar{x})).$$

Combining these inequalities with the Lipschitz continuity and monotonicity of $g$, one has

\begin{align*}
-L_2h+\mu CL_4h^{\frac{2-\sigma}{2}}&\le  g(\bar{x}^{\epsilon,*}+he, U^{\epsilon}(\bar{x}))-g(\bar{x}_{\epsilon,*}, V_\epsilon(\bar{x}))+L_2|\bar{x}^{\epsilon,*}-\bar{x}_{\epsilon,*}|\\& \le L_1\int\frac{|\delta\phi(\bar{x}^{\epsilon,*},y)-\delta\phi(\bar{x}^{\epsilon,*}+he,y)|}{|y|^{n+\sigma}}dy\\&+L_1\int\frac{|\delta\phi(\bar{x}_{\epsilon,*},y)-\delta\phi(\bar{x}^{\epsilon,*},y)|}{|y|^{n+\sigma}}dy+L_2|\bar{x}^{\epsilon,*}-\bar{x}_{\epsilon,*}|\\&=L_1\int\frac{|\delta\phi(\bar{x}^{\epsilon,*},y)-\delta\phi(\bar{x}^{\epsilon,*}+he,y)|}{|y|^{n+\sigma}}dy+o(1)
\end{align*}as $\epsilon\to 0$.

We now estimate the remaining term on the right-hand side. 

For $R>0$,  we can split the integral into 
\begin{equation*}\int_{B_R}\frac{|\delta\phi(\bar{x}^{\epsilon,*},y)-\delta\phi(\bar{x}^{\epsilon,*}+he,y)|}{|y|^{n+\sigma}}dy+\int_{B_R^c}\frac{|\delta\phi(\bar{x}^{\epsilon,*},y)-\delta\phi(\bar{x}^{\epsilon,*}+he,y)|}{|y|^{n+\sigma}}dy.
\end{equation*}

For the first term,
\begin{align*}\int_{B_R}\frac{|\delta\phi(\bar{x}^{\epsilon,*},y)-\delta\phi(\bar{x}^{\epsilon,*}+he,y)|}{|y|^{n+\sigma}}dy&\le \int_{B_R}\frac{|\delta\phi(\bar{x}^{\epsilon,*},y)|}{|y|^{n+\sigma}}dy+\int_{B_R}\frac{|\delta\phi(\bar{x}^{\epsilon,*}+he,y)|}{|y|^{n+\sigma}}dy\\&\le 2L_3\int_{B_R}\frac{|y|^2}{|y|^{n+\sigma}}dy\\&= 2L_3R^{2-\sigma}.
\end{align*}

For the second,
\begin{align*}\int_{B_R^c}\frac{|\delta\phi(\bar{x}^{\epsilon,*},y)-\delta\phi(\bar{x}^{\epsilon,*}+he,y)|}{|y|^{n+\sigma}}dy&\le \int_{B_R^c}\frac{|\phi(\bar{x}^{\epsilon,*}+y)-\phi(\bar{x}^{\epsilon,*}+he+y)|}{|y|^{n+\sigma}}dy\\&+\int_{B_R^c}\frac{|\phi(\bar{x}^{\epsilon,*}-y)-\phi(\bar{x}^{\epsilon,*}+he-y)|}{|y|^{n+\sigma}}dy
\\&+2\int_{B_R^c}\frac{|\phi(\bar{x}^{\epsilon,*})-\phi(\bar{x}^{\epsilon,*}+he)|}{|y|^{n+\sigma}}dy\\&\le 4L_4h\int_{B_R^c}\frac{1}{|y|^{n+\sigma}}dy\\&\le 4L_4hR^{-\sigma}.
\end{align*}
 
 By choosing $R^2=(L_4h/L_3)$, one obtains 
$$\int\frac{|\delta\phi(\bar{x}^{\epsilon,*},y)-\delta\phi(\bar{x}^{\epsilon,*}+he,y)|}{|y|^{n+\sigma}}dy\le C_{n,\sigma}L_4^{\frac{2-\sigma}{2}}L_3^{\frac{\sigma}{2}}h^{\frac{2-\sigma}{2}}.$$

Consequently, $$\mu CL_4h^{\frac{2-\sigma}{2}}\le L_2h+C_{n,\sigma}L_4^{\frac{2-\sigma}{2}}L_3^{\frac{\sigma}{2}}h^{\frac{2-\sigma}{2}}.$$ This leads to a contradiction once we choose $C$ to be sufficiently large, depending on $L_2, L_3, L_4$ and $\mu$. 

\end{proof}


\section{Existence of viscosity solutions}
In this section we prove the existence of viscosity solutions to the Dirichlet problem under some extra assumptions. These assumptions are needed for the construction of a supersolution to our problem. Then we use Perron's method to solve the Dirichlet problem if the operator is uniformly elliptic. The H\"older regularity estimate in the previous section gives enough compactness for us to approximate our degenerate operator with a sequence  of uniform elliptic operators. 

As already commented in Introduction, we have a natural subsolution to the problem:
\begin{lem}
$\phi$ is a classical subsolution to the Dirichlet problem.
\end{lem} 

We also assume that $\phi$ is `close to a cone' at $\infty$. This notion was used in Caffarelli-Charro \cite{CCh} to give decay of $\Delta^{\sigma/2}\phi$, which is crucial for an upper bound on the operator.

\begin{defi}
We say $\phi$ is close to a cone at $\infty$ if $\phi=\Gamma+\eta$ near infinity, where $\Gamma$ is a cone, and for some $C$ and $0<\epsilon<n$, one has $$|\eta(x)|\le C|x|^{-\epsilon}, |\nabla\eta(x)|\le C|x|^{-1-\epsilon}, |D^2\eta(x)|\le C|x|^{-2-\epsilon}.$$
\end{defi} 

\begin{lem}
If $\phi$ is close to a cone at $\infty$, then $$\Delta^{\sigma/2}\phi(x)\le C\min\{1,\frac{1}{|x|^{\sigma-1}}\}.$$
\end{lem} 

\begin{proof}
See Lemma 6.1 in \cite{CCh}.
\end{proof} 

In the following two lemmas, we show the construction of supersolutions under certain extra conditions on $F$ or $g$. We point out that in both cases the supersolution is `universal' in the sense that it does not depend on the derivative of $F$.

\begin{lem}
If $\phi$ is close to a cone at $\infty$ and $g(x,t)/t\to+\infty$ as $t\to+\infty$ locally uniformly in $x$, then there is a $\bar{u}$ continuous and
$$\begin{cases}\int\frac{F(\delta\bar{u}(x,y))}{|y|^{n+\sigma}}dy\le g(x,\bar{u}(x)-\phi(x))&\text{in $\R$}\\ \bar{u}-\phi\to 0&\text{at $\infty$}.\end{cases}$$
\end{lem} 

\begin{proof}
Take $u_0(x)=1/|x|^{p}$ outside $B_1$, and positive, smooth and bounded by $1$ throughout $\R$, where $p>0$ is to be chosen. Then in particular we would have $\int\frac{|\delta u_0(x,y)|}{|y|^{n+\sigma}}dy$ locally bounded. 

Furthermore for large $x$,
\begin{align*}
\int\frac{|\delta u_0(x,y)|}{|y|^{n+\sigma}}dy&=\int_{|y|<|x|/2}\frac{|\delta u_0(x,y)|}{|y|^{n+\sigma}}dy+\int_{|y|>|x|/2}\frac{|\delta u_0(x,y)|}{|y|^{n+\sigma}}dy\\&\le C\frac{1}{|x|^{p+2}}\int_{|y|<|x|/2}\frac{|y|^2}{|y|^{n+\sigma}}dy+\int_{|y|>|x|/2}\frac{4}{|y|^{n+\sigma}}dy\\&\le C\frac{1}{|x|^{p+2}}|x|^{2-\sigma}+C|x|^{-\sigma}\\&\le C/|x|^{\sigma}
\\&\le C/|x|^{\sigma-1}.
\end{align*}

Now define $\bar{u}=\phi+Mu_0$, where $M$ is to be chosen. Then obviously $\bar{u}-\phi\to 0$ at infinity. Also
\begin{align*}
\int\frac{F(\delta\bar{u}(x,y))}{|y|^{n+\sigma}}dy&\le \int\frac{F(\delta\phi(x,y)+M\delta u_0(x,y) )}{|y|^{n+\sigma}}dy\\&\le \int\frac{F(\delta\phi(x,y))}{|y|^{n+\sigma}}dy+L_1\int\frac{|M\delta u_0(x,y)|}{|y|^{n+\sigma}}dy\\&\le L_1\int\frac{\delta\phi(x,y)}{|y|^{n+\sigma}}dy+ML_1\int\frac{|\delta u_0(x,y)|}{|y|^{n+\sigma}}dy,
\end{align*} which decays like $1/|x|^{\sigma-1}$ for large $x$ and remains bounded on compact sets. Therefore if we choose $p<\sigma-1$, then $g(x,\bar{u}-\phi)=g(x,Mu_0)\ge\mu Mu_0(x)=\mu M/|x|^{p}$ dominates it for large $x$.

For small $x$, simply note $\int\frac{F(\delta\bar{u}(x,y))}{|y|^{n+\sigma}}dy\le C+CM,$ this will be dominated by $g$ if we choose very large $M$ and use the superlinearity condition.\end{proof} 

Here is another condition that gives the existence of supersolutions.
\begin{lem}
If $\phi$ is close to a cone at $\infty$ and $F$ is concave, then there is $\bar{u}$ continuous and
$$\begin{cases}\int\frac{F(\delta\bar{u}(x,y))}{|y|^{n+\sigma}}dy\le g(x,\bar{u}(x)-\phi(x))&\text{in $\R$}\\ \bar{u}-\phi\to 0&\text{at $\infty$}.\end{cases}$$
\end{lem} 

\begin{proof}
We define $u_0$ to be the convolution of $min\{1,|x|^{-\sigma-\tau}\}$ and the fundamental solution to $-\Delta^{\sigma/2}$, where $\tau<\min\{\sigma-1,n-\sigma\}.$  It is shown in Lemma 6.1 in \cite{CCh} that $$u_0(x)\ge C\min\{1,|x|^{-\tau}\}.$$
Also it's clear $\Delta^{\sigma/2}u_0\le 0$ and $u_0$ vanishes at infinity.

Now if we define $\bar{u}=\phi+Mu_0$, then $\bar{u}-\phi\to 0$ at infinity. Moreover, 
\begin{align*}
\int\frac{F(\delta\bar{u}(x,y))}{|y|^{n+\sigma}}dy&=\int\frac{F(\delta\phi(x,y)+M\delta u_0(x,y))}{|y|^{n+\sigma}}dy\\&\le \int\frac{F(M\delta u_0(x,y))}{|y|^{n+\sigma}}dy+L_1\int\frac{\delta\phi(x,y)}{|y|^{n+\sigma}}dy\\&\le MF'(0)\int\frac{\delta u_0(x,y)}{|y|^{n+\sigma}}dy+L_1\int\frac{\delta\phi(x,y)}{|y|^{n+\sigma}}dy\\&\le L_1\int\frac{\delta\phi(x,y)}{|y|^{n+\sigma}}dy\\&\le C\min\{1,|x|^{1-\sigma}\}.
\end{align*}Note that we used both the concavity of $F$ and the convexity of $\phi$ in this estimate.

Meanwhile, $g(x,\bar{u}-\phi)\ge \mu Mu_0(x)=C\mu M\min\{1,|x|^{-\tau}\}$. Thus a large $M$ gives the desired inequality.
\end{proof} 

Once we have appropriate sub- and super-solutions, the existence of solution follows from the standard Perron's method, at least for uniformly elliptic operators. 

\begin{prop}
Suppose $\phi$ is close to a cone at $\infty$, and $0<\lambda<F'<\Lambda<+\infty$. 
Then under either of these assumptions \begin{itemize}
\item{$F$ is concave}
\end{itemize} or \begin{itemize}\item{$g$ grows superlinearly in the second variable,}\end{itemize}
there exists a unique viscosity solution to the Dirichlet problem.
\end{prop} 

\begin{proof}Uniqueness follows from Theorem 3.1.

Let $\bar{u}$ denote the supersolution in either case. Define $$u=\sup\{w \text{ subsolution}| \phi\le w\le\bar{u}\}.$$ It is clear that $u$ is well-defined and satisfies the boundary condition. 

Also standard elliptic theory guarantees $u$ is a viscosity subsolution.

Suppose, on the contrary, that $u$ fails to be a supersolution, then one finds a locally $C^{1,1}$ function $\eta$ touching $u$ from below at $x_0$ but $$\int\frac{F(\delta\tilde{\eta}(x_0,y))}{|y|^{n+\sigma}}dy>g(x_0,u_0(x_0)-\phi(x_0))+\delta,$$ where $\tilde{\eta}$ is the function agreeing with $\eta$ in a neighborhood of $x_0$ and agrees with $u$ outside, and $\delta$ is some positive number. Note that by adding a higher order perturbation we might assume $\eta$ is strictly above $u$ in that neighborhood other than at $x_0$. By doing this we can still make sure $\eta$ stays below $\bar{u}$ since $\bar{u}$ is a classical supersolution while $\eta$ is a strict subsolution in that neighborhood.

But being in the uniform elliptic regime the operator is continuous near $x_0$. We can thus replace $u$ with $\tilde{\eta}$ in a small neighborhood and obtain a subsolution greater than $u$, a contradiction.

\end{proof} 

Now we prove the following theorem by the previous result and an approximating procedure.

\begin{thm}Suppose $\phi$ is close to a cone at $\infty$. Then under either of these assumptions \begin{itemize}
\item{$F$ is concave}
\end{itemize} or \begin{itemize}\item{$g$ grows superlinearly in the second variable,}\end{itemize}
there exists a unique viscosity solution to the Dirichlet problem.\end{thm} 

\begin{proof}Again uniqueness is due to the comparison.

For each $\epsilon>0$ we define an approximating nonlinearity $$F_{\epsilon}(t)=\int_0^t\max\{\epsilon,F'(s)\}ds.$$ These are uniformly elliptic operators, and the previous proposition gives a unique solution $u_\epsilon$ to the Dirichlet problem associated with $F_\epsilon$.

Since our sub- and super- solutions are independent of $\epsilon$, $$\phi\le u_\epsilon\le \bar{u}$$ for all $\epsilon$. And in particular $\|u_\epsilon-\phi\|_{\mathcal{L}^{\infty}}\le \|\phi-\bar{u}\|_{\mathcal{L}^{\infty}}$ uniformly in $\epsilon$. Theorem 3.3 then gives the equicontinuity of this family $\{u_\epsilon\}$. Thus we can extract a subsequence $\epsilon\to 0$ so that $u_\epsilon\to u$ locally uniformly in $\R$.

We now verify that $u$ is the desired solution to the degenerate Dirichlet problem.

Suppose, on the contrary, that $u$ fails to be a subsolution. Then we can find 
$\eta$, a smooth function touching $u$ from above at $x_0$ in some open set $U$, but 
$$\int \frac{F(\delta\tilde{\eta}(x_0,y))}{|y|^{n+\sigma}}dy<g(x_0,\tilde{\eta}(x_0)-\phi(x_0))-\gamma=g(x_0,u(x_0)-\phi(x_0))-\gamma,$$ where $\gamma>0$ and 
$$\tilde{\eta}=\begin{cases} \eta &\text{in $U$}\\u&\text{outside $U$}.\end{cases}$$ 

Without loss of generality we assume $x_0=0$.  

Define $\psi=\eta+\delta|x|^4$ for some $\delta>0$. Then $\psi$ touches $u$ from above at $0$ in $U$. 

For each $m\in\mathbb{N}$ let $\epsilon_m$ be small so that $$u+\delta(1/m)^4/16>u_{\epsilon_m}$$ in $B_1$. Then with $\psi\ge u+\delta(1/m)^4/16$ outside $B_{\frac{1}{2m}}$, we can find $b_m$ such that $\psi+b_m$ touches $u_{\epsilon_m}$ from above at some point $x_m\in B_{\frac{1}{2m}}$. In particular $b_m=o(1)$ and $x_m\to 0$ as $m\to\infty$.

To use the equation  we define 
$$\tilde{\psi}=\begin{cases} \psi &\text{in $U$}\\u &\text{outside $U$}.\end{cases}$$and
$$\tilde{\psi}_m=\begin{cases} \psi+b_m &\text{in $U$}\\u_{\epsilon_m} &\text{outside $U$}.\end{cases}$$

Then one has \begin{align*}\int \frac{F_{\epsilon_m}(\delta\tilde{\psi}_m(x_m,y))}{|y|^{n+\sigma}}dy&\ge g(x_m,(u_{\epsilon_m}-\phi)(x_m))\\&=g(x_m,\psi(x_m)+b_m-\phi(x_m))\\&=g(0,\psi(0)-\phi(0))+o(1)\\&=g(0,u(0)-\phi(0))+o(1)\\&>\int \frac{F(\delta\tilde{\eta}(0,y))}{|y|^{n+\sigma}}dy+\gamma+o(1).\end{align*}

On the other hand, by definition of the approximating operators, $$|F_{\epsilon}(t)-F(t)|\le \epsilon|t|.$$As a result,\begin{align*}\int \frac{F_{\epsilon_m}(\delta\tilde{\psi}_m(x_m,y))}{|y|^{n+\sigma}}dy&\le \int \frac{F(\delta\tilde{\psi}_m(x_m,y))}{|y|^{n+\sigma}}dy+\epsilon_m\int\frac{|\delta \tilde{\psi}_m(x_m,y)}{|y|^{n+\sigma}}dy\\&\le \int \frac{F(\delta\tilde{\psi}(0,y))}{|y|^{n+\sigma}}dy+o(1)+\epsilon_m O(1).
\end{align*}The last inequality can be justified by dominated convergence theorem, and noting that all the test functions are uniformly bounded in $\mathcal{L}^1(\frac{1}{(1+|y|)^{n+\sigma}}dy)$, and they all inherit the same $C^{1,1}$ constant from $\eta$. 

Now by sending $m\to\infty$ and $\delta\to 0$, we have 
$\gamma<o(1)$, a contradiction. As a result $u$ is a subsolution. 

Similar argument shows $u$ is also a supersolution.

\end{proof}

\section{Regularity of the solution}
In this section we study the regularity of the solution we obtained in the previous section. In particular we will assume $\phi\le u\le\bar{u}$, where $\bar{u}$ is some proper supersolution. By Theorem 3.3 we know that $u$ is also H\"older continuous. To get higher regularity,  we first identify some conditions under which our solution is classical. Then we show we can bootstrap to smoothness once the operator, the forcing term as well as the boundary datum are all smooth.

To fit into existing theory, we further impose the following uniform ellipticity condition on $F$ throughout this section \begin{equation}
0<\lambda<F'<\Lambda<\infty.
\end{equation} 

The starting point is to find the equation satisfied by $w=u-\phi$.

\begin{lem}
$w$ satisfies, in the viscosity sense, \begin{equation}\begin{cases}\int\frac{F(\delta w(x,y))}{|y|^{n+\sigma}}dy=g(x,w(x))-\int\delta\phi(x,y)\frac{a(x,y)}{|y|^{n+\sigma}}dy &\text{in $\R$}\\w=0&\text{at $\infty$},\end{cases}\end{equation}where $$a(x,y)=\int_0^1F'(\delta u(x,y)-t\delta\phi(x,y))dt.$$
\end{lem} 

\begin{proof}$w$ clearly vanishes at infinity.

Let $\eta$ be a smooth function touching $w$ from above at $x_0$ in some open set $U$. 

Define 
$$\tilde{\eta}=\begin{cases}\eta&\text{in $U$}\\w&\text{outside $U$}\end{cases}.$$ Then $\tilde{\eta}+\phi$ is a test function for $u$ touching $u$ from above at $x_0$.

Since $u$ is a viscosity solution one has $$\int\frac{F(\delta(\tilde{\eta}+\phi)(x_0,y))}{|y|^{n+\sigma}}dy\ge g(x_0,(\tilde{\eta}+\phi)(x_0)-\phi(x_0))=g(x_0,w(x_0)).$$

Meawhile, Fundamental Theorem of Calculus gives 
\begin{align*}F(\delta(\tilde{\eta}+\phi)(x_0,y))&=F(\delta\tilde{\eta}(x_0,y))+\int_0^1\frac{d}{dt}F(\delta(\tilde{\eta}+t\phi)(x_0,y))dt\\&=F(\delta\tilde{\eta}(x_0,y))+\int_0^1 F'(\delta(\tilde{\eta}+t\phi)(x_0,y))dt\cdot\delta\phi(x_0,y)\\&=F(\delta\tilde{\eta}(x_0,y))+a(x_0,y)\delta\phi(x_0,y).\end{align*}

Consequently, 
$$\int\frac{F(\delta \tilde{\eta}(x_0,y))}{|y|^{n+\sigma}}dy=g(x_0,w(x_0))-\int\delta\phi(x_0,y)\frac{a(x_0,y)}{|y|^{n+\sigma}}dy$$ and $w$ is a subsolution. Similarly $w$ is also a supersolution.
\end{proof} 

The following theorem identifies some conditions to guarantee that the solution is classical:

\begin{thm}
If \begin{itemize}\item{$F$ is $C^2$ with $0<\lambda <F'<\Lambda$ and $|F''|\le L_6$,}\item{$g$ is locally Lipschitz in the second variable,} and \item{$Lip(D^2(\phi))<L_7$,}
\end{itemize} then the viscosity solution $u$ is in the class  $C^{1+\sigma+\alpha}$ for some universal $0<\alpha<1$, and in particular is classical.
\end{thm} 

\begin{proof}
Let's first assume $w\in C^{\beta}$ for some $\beta$. For instance, we know this is true for $\beta=(2-\sigma)/2$. Denote the right-hand side of (5.2) by $G$, then for $h\in\mathbb{R}$ and $e\in\mathbb{S}^{n-1}$, $G$ satisfies \begin{equation}|G(x+he)-G(x)|\le C|h|^{\beta}.\end{equation}
To see this, note that 
\begin{align*}|g(x+he,w(x+he))-g(x,w(x))|&\le L_2|h|+C|w(x+he)-w(x)|\\&\le L_2|h|+C|h|^{\beta}\\&\le C|h|^{\beta}.
\end{align*} Here we used the fact that $w$ is uniformly bounded, and hence $g$ being locally Lipchitz is as good as being globally Lipschitz in the second variable.
\begin{align*}|a(x+he,y)-a(x,y)|&=|\int_0^1(F'(\delta u(x+he,y)-t\delta\phi(x+he,y))-F'(\delta u(x,y)-t\delta\phi(x,y)))dt|\\&\le \|F''\|(|\delta u(x+he,y)-\delta u(x,y)|+|\delta\phi(x+he,y)-\delta\phi(x,y)|)\\&\le C|h|^{\beta}.
\end{align*}Here we used the Lipchitz regularity of $\phi$.

Thus 
\begin{align*}
|\int (a(x+he,y)-a(x,y))\frac{\delta\phi(x,y)}{|y|^{n+\sigma}}dy|&\le C|h|^{\beta}\int\frac{\delta\phi(x,y)}{|y|^{n+\sigma}}dy\\&\le C|h|^{\beta}.
\end{align*} This is due to the $C^{1,1}$ and $\mathcal{L}^1(\frac{1}{(1+|y|)^{n+\sigma}}dy)$ bound of $\phi$.

Moreover, \begin{align*}
|\int_{B_1^c}\frac{a(x,y)}{|y|^{n+\sigma}}(\delta\phi(x+he,y)-\delta\phi(x,y))dy|&\le 4\int_{B^c_1}\frac{1}{|y|^{n+\sigma}}dy|h|\\&\le C|h|\\&\le C|h|^{\beta}.\end{align*} Again the Lipschitz continuity of $\phi$ is needed here.

And finally 
\begin{align*}|\int_{B_1} \frac{a(x,y)}{|y|^{n+\sigma}}(\delta\phi(x+he,y)-\delta\phi(x,y))dy|&\le C\int_{B_1}\frac{|y|^2}{|y|^{n+\sigma}}dy\cdot L_7|h|\\&\le C|h|^{\beta}.
\end{align*}

(5.3) follows from these four inequalities.

Once (5.3) is established, we can define the difference quotient $$v(x)=\frac{w(x+he)-w(x)}{|h|^\beta}.$$It satisfies 
$M^+_0v(x)\ge \frac{1}{|h|^\beta}(G(x+he)-G(x))$ and $M^-_0v(x)\le \frac{1}{|h|^\beta}(G(x+he)-G(x))$. Here we are using notations from \cite{CS1}. 

Based on (5.3) the right-hand side of these inequalities are bounded functions, and the H\"older estimate in \cite{CS1} gives $v\in C^{\alpha}$ for some universal $\alpha>0$, and hence $w\in C^{\beta+\alpha}$. 

This argument can be applied finitely many times to show $w\in C^{1+\alpha}$.

But once we know $w\in C^{1+\alpha}$, we know $G$ is actually Lipschitz.  Then Theorem 7.2 in \cite{Kri} gives $w\in C^{1+\sigma+\alpha}$. Note that Theorem 7.2 is formulated for smooth right-hand side, but it is clear from the proof that the estimates only depend on the Lipschitz semi-norm of the right-hand side. 
\end{proof} 

Once the solution is classical, we can bootstrap to smooth solutions. We'd like to point out that this bootstrap argument fails for equations in a bounded domain due to very rough boundary behaviour of nonlocal equations. See \cite{RosSer} and \cite{Y1}. However this is not a problem for us since we are dealing with an entire solution.  Therefore the following theorem actually indicates that the rough boundary behaviour might be  the only obstruction to a bootstrap argument for nonlocal equations.

\begin{thm}
If we further assume $F$, $g$ and $\phi$ are smooth, then $u$ is smooth.
\end{thm} 

\begin{proof}
Now we have a classical solution we can differentiate the equation for $u$ in $e$ (actually even the derivatives of $u$ have enough regularity for a $\sigma$-order operator, since they are in $C^{\sigma+\alpha}$) to get 
$$\int\delta u_e(x,y)\frac{F'(\delta u(x,y))}{|y|^{n+\sigma}}dy=\nabla_xg(x,u(x)-\phi(x))\cdot e +\partial_t g(x,u(x)-\phi(x))(u_e(x)-\phi_e(x)).$$ Note that the right-hand side is $C^{\sigma+\alpha}$.
Differentiate once more to get 
\begin{align*}\int\delta u_{ee}(x,y)\frac{F'(\delta u(x,y))}{|y|^{n+\sigma}}dy=&\partial_e(\nabla_xg(x,u(x)-\phi(x))\cdot e +\partial_t g(x,u(x)-\phi(x))(u_e(x)-\phi_e(x)))\\&-\int(\delta u_e(x,y))^2\frac{F''(\delta u(x,y))}{|y|^{n+\sigma}}dy.\end{align*} This is a linear elliptic equation with H\"older coefficient and H\"older right-hand side, thus Schauder theory \cite{JX}\cite{Ser} gives $u\in C^{2+\sigma+\alpha}$.

From here it is clear how this argument can be iterated.
\end{proof}

\section*{Acknowledgement} The author would like to thank his PhD advisor, Luis Caffarelli, for many valuable conversations regarding this project. He is also grateful to his colleagues and friends, especially Dennis Kriventsov, Xavier Ros-Oton and Tianlin Jin  for all the discussions and encouragement. 


\end{document}